\documentclass[reqno]{amsart}
\begin{document}

\newtheorem{theorem}{Theorem}[section]
\newtheorem{thm}[theorem]{Theorem}
\newtheorem{prop}[theorem]{Proposition}
\newtheorem{lemma}[theorem]{Lemma}

\theoremstyle{definition}
\newtheorem{definition}[theorem]{Definition}
\newtheorem{remark}[theorem]{Remark}

\newcommand{\C}{{\mathbb C}}
\newcommand\BB{{\mathcal B}}
\newcommand\FF{{\mathcal F}}
\newcommand\VV{{\mathcal V}}
\newcommand\WW{{\mathcal W}}
\newcommand\Z{{\mathbb Z}}
\newcommand{\bbar}{\overline}

\title[An almost complex Castelnuovo de Franchis theorem]{An almost complex
Castelnuovo de Franchis theorem}

\author{Indranil Biswas}

\address{School of Mathematics, Tata Institute of Fundamental Research, Homi 
Bhabha Road, Mumbai 400005, India}

\email{indranil@math.tifr.res.in}

\author{Mahan Mj}

\address{School of Mathematics, Tata Institute of Fundamental Research, Homi 
Bhabha Road, Mumbai 400005, India}

\email{mahan@math.tifr.res.in}



\date{}

\begin{abstract}
Given a compact almost complex manifold, we prove
a Castelnuovo--de Franchis type theorem for it.
\end{abstract}

\maketitle

\section{Introduction}

Given a smooth complex projective variety $X$, the classical Castelnuovo-de 
Franchis theorem \cite{cas, df} associates to an isotropic subspace of $H^0(X,\, 
\Omega^1_X)$ of dimension greater than one an irrational pencil on $X$. The 
topological nature of this theorem was brought out by Catanese \cite{cat}, who 
established a bijective correspondence between
subspaces $\widetilde{U} \,\subset\, H^1(X, \C)$
of the form $U \oplus \bbar{U}$ with $U$ being a maximal isotropic subspace of $H^1(X,
\C)$ of dimension $g\, \geq\, 1$ and irrational fibrations on $X$ of genus $g$.
The purpose of this note is to emphasize this topological content further. We 
extract topological hypotheses that allow the theorem to go through when we have 
only an almost complex structure.

In what follows, $M$ will be a compact smooth manifold of dimension $2k$. Let
$$
J_M\, :\, TM\,\longrightarrow\, TM
$$
be an almost complex structure on $M$, meaning $J_M\circ J_M\,=\, -\text{Id}_{TM}$.

\begin{definition}
We shall say that a collection of closed complex $1$--forms $\omega_1, \cdots,
\omega_n$ on $M$ are in {\it general position} if 
\begin{enumerate}
\item the zero-sets 
$$
Z(\omega_i)\,:=\, \{x\, \in\, M\, \mid\, \omega_i(x)\,=\, 0\}\, \subset\, M
$$
are smooth embedded submanifolds, and
\item these submanifolds $Z(\omega_i)$ intersect transversally.
\end{enumerate}
\end{definition}

We are now in a position to state the main theorem of this note.

\begin{theorem} \label{main}
Let $M$ be a compact smooth $2k$--manifold equipped with an almost complex structure
$J_M$. Let $\omega_1\, , \cdots\, , \omega_g$
be closed complex $1$--forms on $M$ linearly independent over $\C$, with $g\, \geq\, 2$,
such that
\begin{itemize}
\item each $\omega_i$ is of type $(1\, ,0)$,
meaning $\omega_i(J_M (v))\,=\, \sqrt{-1}\cdot\omega_i(v)$ for all $v\,\in\, TM$,
and

\item $\omega_i$ are in general position with $\omega_i \wedge \omega_j \,=\, 0$ for all
$i\, , j$.
\end{itemize}
Then there exists a smooth almost holomorphic map 
$f \,:\, M\,\longrightarrow\, C$ to a compact Riemann surface of genus at least $g$, and
there are linearly independent holomorphic $1$--forms 
$\eta_1\, , \cdots\, , \eta_g$ on $C$, such that $\omega_i \,=\, f^\ast \eta_i$
for all $i$.
\end{theorem}

\section{Leaf space and almost complex blow-up}\label{se2}

\subsection{Leaf space}

Assume that $\omega_i$ are forms as in Theorem \ref{main}. Since $\omega_i \wedge 
\omega_j \,=\, 0$ for all $1\,\leq\, i\, , j\,\leq\, g$, it follows that there are 
complex valued smooth functions $f_{i,j}$ such that
\begin{equation}\label{e1}
\omega_i \,=\, f_{i,j} \omega_j
\end{equation} 
wherever $\omega_j \,\neq\, 0$. Hence the collection
$$\WW\,=\,\{ \omega_1, \cdots, \omega_g \}$$
determines a complex line subbundle of the complexified cotangent bundle
$(T^*M)\otimes{\mathbb C}$ over the open subset
$$
V\,:=\, M \setminus \bigcap_{i=1}^g Z(\omega_i)\,\subset\, M\, .
$$ 

\begin{lemma}\label{cod2}
Let $\FF\,:=\, \{v\, \in\, TV\,\mid\, \omega_i(v)\,=\, 0\,~~
\forall ~~ 1\,\leq\, i\, \leq\, g\}\,\subset\, TV$ be the
distribution on $V$ defined by $\WW$. Then $\FF$ is integrable and 
defines a foliation of real codimension two on $V$.
\end{lemma}

\begin{proof}
For any $x\, \in\, M$ and $1\,\leq\, i\, \leq\, g$, consider the $\mathbb R$--linear
homomorphism
$$
\omega^x_i\, :\, T_xM \, \longrightarrow\,
{\mathbb C}\, ,~~~\, ~~ v\, \longmapsto\, \omega_i(x)(v)\, .
$$
Since $\omega_i$ is of type $(1\, ,0)$, if $\omega_i(x)\, \not=\, 0$, then
$\omega^x_i$ is surjective. Also, $\omega^x_j$ is a scalar multiple of
$\omega^x_i$ because $\omega_i\wedge \omega_j\,=\, 0$. Therefore, we conclude that
the distribution $\FF$ on $V$ is of real codimension two.

Since $dw_i\,=\, 0$ for all $i$, it follows that the distribution $\FF$ is
integrable.
\end{proof}

\begin{remark}\label{even}
Since $Z(\omega_i)$ are smooth embedded submanifolds by hypothesis,
and $\omega_i$ are of type $(1\, ,0)$, it follows 
that $Z(\omega_i)$ are almost complex submanifolds of $M$, meaning $J_M$ preserves the 
tangent subbundle $TZ(\omega_i)\, \subset\, (TM)\vert_{Z(\omega_i)}$.
Consequently, all the intersections of the $Z(\omega_i)$'s are also almost complex
submanifolds and are therefore even dimensional.

Clearly, each $Z(\omega_i)$ has complex codimension at least one (real codimension at 
least two) as an almost complex submanifold. Therefore, the complement
$M \setminus V$ has complex codimension at least two.
\end{remark}

\subsection{Almost complex blow-up}

We shall need an appropriate notion of blow-up in our context. Suppose 
$K\,\subset\, M$ is a smooth embedded submanifold of $M$ of dimension $2j$ such 
that $J_M (TK) \,=\, TK$. By Remark \ref{even}, the submanifold $\bigcap_{i=1}^g 
Z(\omega_i)$ satisfies these conditions. Note that
$J_M$ induces an automorphism
$$
J_{M/K}\, :\, ((TM)\vert_K)/TK\,\longrightarrow\, ((TM)\vert_K)/TK
$$
of the quotient bundle over $K$. Since $J_M$ is an almost complex structure it
follows that $J_{M/K}\circ J_{M/K}\,=\, - \text{Id}_{((TM)\vert_K)/TK}$. 
Therefore, $((TM)\vert_K)/TK$ is a complex vector bundle on $K$ of rank
$k-j$. We would like to replace $K$ by the (complex) projectivized normal bundle
${\mathbb P}(((TM)\vert_K)/TK)$ which will be called the \textbf{almost
complex blow-up} of $M$ along $K$.

For notational convenience, the intersection $\bigcap_{i=1}^g Z(\omega_i)$ will be 
denoted by $\mathcal Z$.

We first projectivize $(TM)\vert_{\mathcal Z}$ to get a 
${\mathbb C}{\mathbb P}^{k-1}$ bundle over $\mathcal Z$; this $\C{\mathbb P}^{k-1}$ 
bundle will be denoted by $\BB$. So $\BB$ parametrizes the space of all (real) two 
dimensional subspaces of $(TM)\vert_{\mathcal Z}$ preserved by $J_M$ (such two 
dimensional subspaces are precisely the complex lines in $(TM)\vert_{\mathcal Z}$ 
equipped with the complex vector bundle structure defined by $J_M$). Let
$$\pi\,:\, \BB\,\longrightarrow\, \mathcal Z$$
be the natural projection.

For notational convenience, the pulled back vector bundle $\pi^\ast 
((TM)\vert_{\mathcal Z})$ will be denoted by $\pi^\ast TM$.

Let
$$
T_\pi\, :=\, {\rm kernel}(d\pi)\, \subset\, T\BB
$$
be the relative tangent bundle, where $d\pi\, :\, T\BB\, \longrightarrow\,
\pi^\ast T{\mathcal Z}$ is the differential of $\pi$. The map $\pi$ is almost
holomorphic, and $T_\pi$ has the structure of a complex vector bundle.
It is known that the complex vector bundle $T_\pi$ is identified
with the vector bundle $Hom_{\mathbb C}({\mathcal L},\, (\pi^\ast TM)/{\mathcal L})\,=\,
((\pi^\ast TM)/{\mathcal L})\otimes_{\mathbb C} {\mathcal L}^*$, where
$$
{\mathcal L}\, \subset\, \pi^\ast TM
$$
is the tautological real vector bundle of rank two; note that both ${\mathcal L}$ 
and $(\pi^\ast TM)/{\mathcal L}$ have structures of complex vector bundles given by 
$J_M$, and the above tensor product and homomorphisms are both over $\mathbb C$. 
Therefore, the pullback $\pi^\ast TM$ splits as $${\mathcal L}\oplus ((\pi^\ast 
TM)/{\mathcal L}) \,=\, {\mathcal L}\oplus (T_{\pi}\otimes {\mathcal L})\, .
$$
The image of the zero section of the complex line bundle ${\mathcal L}
\,\longrightarrow\, \BB$ is identified with 
$\BB$, and the normal bundle of $\BB\, \subset\, {\mathcal L}$ is identified
with ${\mathcal L}$. A small deleted normal neighborhood $U_{\BB}$ of $\BB$ in
${\mathcal L}$ 
can be identified with a deleted neighborhood $U$ of $\mathcal
Z$ in $M$. Let $\overline{U}_{\BB}\,:=\, U_{\BB}\bigcup \BB$ be the
neighborhood of $\BB$ in ${\mathcal L}$ (it is no longer a deleted neighborhood),
where $\BB$ is again identified with the image of the zero section of $\mathcal L$.
In the disjoint union $(M\setminus{\mathcal Z})\sqcup 
\overline{U}_{\BB}$, we may identify $U$ with $U_{\BB}$. The resulting topological 
space will be called the {\bf almost complex blow-up} of $M$ along 
${\mathcal Z}$.

\begin{remark}\mbox{}
\begin{enumerate}
\item Consider the foliation on the deleted neighborhood $U$ of $\mathcal Z$ in $M$
defined by $\mathcal F$. Let ${\mathbb L}_U$ denote the leaf space for it.
After identifying $\BB$ with the image of the zero section of $\pi^\ast TM$, we obtain 
locally, from a neighborhood of $\BB$, a map to the leaf space ${\mathbb L}_U$.

\item Note that the notion of an almost complex blow-up above is 
well-defined up to the choice of an identification of $U$ with $U_{\BB}$. Therefore,
the construction yields a well-defined almost complex manifold, up to an isomorphism.
\end{enumerate}
\end{remark}

\section{Proof of Theorem \ref{main}}

We will first show that the leaves of $\FF$ in Section \ref{se2} are proper embedded
submanifolds of $V\,=\, M\setminus {\mathcal Z}$.

Suppose some leaf ${\mathbb L}_0$ does not satisfy the above property. Then there 
exists $x \,\in \,V$ and a neighborhood $U_x$ of $x$ such that $U_x\bigcap{\mathbb 
L}_0$ contains infinitely many leaves (of $\FF$) accumulating at $x$. Let $\sigma$ 
be a smooth path starting at $x$ transverse to $\FF$. Thus for every $\epsilon\,> 
\,0$ there exist
\begin{enumerate}
\item 
distinct (local) leaves $F_1\, , F_2\,\subset\, U_x \bigcap \FF$, such that
(globally) $F_1, F_2 \,\subset\, {\mathbb L}_0$,
\item points $y_j \,\in\, F_j$, $j\,=\, 1\, ,2$, 
\item a sub-path $\sigma_{12}\,\subset\,\sigma
\,\subset\, U_x$ joining $y_1$ and $y_2$, and 
\item a path $\tau_{12}\,\subset\, L_0$ joining $y_1$ and $y_2$ such that
$$\vert\int_{\sigma_{12}} 
\omega_i\vert \,< \,\epsilon\, \ \ ~ \forall\ i\, .
$$
\end{enumerate} 
Since $\FF$ is in the kernel of each $\omega_i$,
$$\vert\int_{\tau_{12}} 
\omega_i\vert \, = \,0, \ \ ~ \forall\ i\, .
$$
Setting
$\epsilon$ smaller than the absolute value of any non-zero period of the $\omega_i$'s,
 it follows that $$\int_{\tau_{12}\cup\sigma_{12}} 
\omega_i\, =\,0~\ \ 1\,\leq\, i\, \leq\, g\, .$$
Hence $$\int_{\sigma_{12}} 
\omega_i\, =\,0~\ \ \forall \ \ 1\,\leq\, i\, \leq\, g\, .$$
Taking limits we obtain
that $\omega_i (\sigma^\prime (0)) \,= \,0$; but this contradicts the 
choice that $\sigma$ is transverse to $\FF$. Therefore, the leaves of $\FF$ are
proper embedded submanifolds of $V$.

Consequently, the leaf space $$D\,=\, V/\langle \FF\rangle$$
for the foliation $\FF$ is a smooth 2-manifold. Let
$$q\,:\, V\,\longrightarrow\, D$$ be the quotient map.

\begin{remark} \label{referee}
The referee kindly pointed out to us the following considerably simpler proof of the 
fact that leaves of $\FF$ are closed in $V$:

By equation \eqref{e1}, we have $\omega_i 
\,=\, f_{i,j} \omega_j$. Since $\omega_i$ are closed, $$d(f_{i,j}) \wedge \omega_j \,=\, 
0, \ \ \forall\ i\, , j\, .$$ The functions $f_{i,j}$ define a map $f\,:\, V
\,\longrightarrow\, \C P^1$. Since 
$d(f_{i,j}) \wedge \omega_j \,=\, 0,$ it follows that $f$ is constant on the leaves of 
the foliation $\FF$ defined by the forms $\omega_i$, $1\,\leq\, i\, leq\, n$.
\end{remark}

Next, let $$\xi \,:=\, TV/ \FF$$ be the quotient complex line bundle on $V$. Then
$\xi$ carries a natural flat partial connection $\mathcal D$ along the
leaves (cf. \cite{La}); this
$\mathcal D$ is known as the Bott partial connection. It is straightforward to check
that $\mathcal D$ preserves the complex structure on $\xi$. Indeed, this follows
immediately from the fact that $\omega_i$ are of type $(1\, ,0)$. Hence $\xi$ induces
an almost complex structure on $D$. Since $D$ is two--dimensional, this almost complex
structure is integrable
giving $D$ the structure of a Riemann surface. Further, there exist closed integral
$(1\, ,0)$ forms $\eta_i$, $1\,\leq\, i\,\leq\, g$, on $D$ such that $$\omega_i\,=
\,f^\ast \eta_i\, .$$

\subsection{Removing indeterminacy}
Since $Z(\omega_i)$ are mutually transverse, we can construct the
almost complex blow-ups of $M$ along $\mathcal Z$ successively along the $\mu$--fold
intersections,
$\mu\,=\,2\, , \cdots\, , g$, to obtain $\widehat M$ such that
\begin{enumerate}
\item there is an extension of the line bundle $\xi$ to all of $\widehat M$,
\item there is a well-defined smooth map
$$\widehat{q}\,:\, \widehat{M} \,\longrightarrow\, D$$ extending $q$, and
\item the blown-up locus has the structure of complex analytic
$\C{\mathbb P}^1$--bundles as usual (due to transversality).
\end{enumerate}

The Riemann surface $D$ is compact because $\widehat M$ is so.
Further, $D$ has genus greater than one as $g \,> \,1$. So
any complex analytic map from $\C{\mathbb P}^1$ to $D$ must be constant. Therefore,
$\widehat q$ actually induces
a smooth map from $M$ to $D$, and we may assume that the indeterminacy locus of $q$ was
empty to start off with,
or equivalently that $q$ extends to a smooth map $\overline{q} \,:\, M\,\longrightarrow\,
D$. This furnishes the required conclusion and completes the proof of Theorem \ref{main}.

\section{Refinements and consequences}

\subsection{Stein factorization}

We now proceed as in the proof of the classical Castelnuovo-de Franchis Theorem,
\cite[p. 24, Theorem 2.7]{abc},
to deduce {\it a posteriori} that $D$ is a Stein factorization of a map to a compact
Riemann surface. Define $$h \,:\, V\,\longrightarrow\, \C{\mathbb P}^{g-1}\, ,
~\ \ ~ m\,\longmapsto\, [\omega_1 (m)\, :\, \cdots \,:\, \omega_g (m)]\, .$$
Suppose $\omega_i(m) \,\neq\, 0$. Then there exists a small neighborhood $U(m)$ such that
$$h(x)\,=\, [f_{1,i} (x)\,:\, \cdots \,:\,f_{i-1,i} (x)\,:\, 1\, :\, f_{i+1,i} (x)
\,: \,\cdots \,:\,f_{g,i} (x)]\, , ~\ \forall\ x \,\in\, U(m)\, ,$$
where $f_{l,j}$ are defined in \eqref{e1}. Since $\omega_j \wedge \omega_i \,=\, 0$
and $d\omega_j\,=\, 0$ for all $i\, ,j$, it follows that $df_{j,i} \wedge \omega_i
\,= \,0$. Consequently, each $f_{j,i}$
is constant on the leaves of $\FF$. Hence $h$ has a (complex) one dimensional image, so
the image of $h$ is a Riemann surface $C$.

Further, $h$ induces a holomorphic map $$h_1\,:\, D\,\longrightarrow\, C\, .$$ We note 
that $D$ may be thought of as the Stein factorization of $h$, and $h$ factors as 
$h\,=\, h_1\circ q$.

\subsection{Further generalizations}

Let $M$ be a compact manifold, and let ${\mathcal S}\,\subset\, TM$ be a nonsingular
foliation such that $M$ has a flat transversely almost complex structure. This means
that we have an automorphism
$$
J\, :\, TM/{\mathcal S}\, \longrightarrow\, TM/{\mathcal S}
$$
such that
\begin{enumerate}
\item $J\circ J\,=\, -\text{Id}_{TM/{\mathcal S}}$, and

\item $J$ is flat with respect to the Bott partial connection on $TM/{\mathcal S}$.
\end{enumerate}
Let $\omega_i$, $1\,\leq\, i\, \leq\, g$, be linearly independent smooth sections of
$(TM/{\mathcal S})^*\otimes{\mathbb C}$ such that
\begin{enumerate}
\item each $\omega_i$ is flat with respect to the connection on
$(TM/{\mathcal S})^*\otimes{\mathbb C}$ induced by the Bott partial connection
on $TM/{\mathcal S}$,

\item each $\omega_i$ is of type $(1\, ,0)$, meaning the corresponding homomorphism
$TM/{\mathcal S}\, \longrightarrow\, M\times\mathbb C$ is $\mathbb C$--linear,

\item each $\omega_i$ is closed when considered as a complex $1$--form on $M$ using
the composition
$$
TM\otimes{\mathbb C}\,\longrightarrow\, (TM/{\mathcal S})\otimes{\mathbb C}
\,\stackrel{\omega_i}{\longrightarrow}\, M\times\mathbb C
$$
\item $\omega_i\wedge\omega_j\,=\, 0$ for all $i\, ,j$, and

\item $\omega_j$ are in general position.
\end{enumerate}
Theorem \ref{main} can be generalized to this set--up.

\begin{remark} 
The only real use we have made of the hypothesis in Theorem \ref{main} that 
$\omega_i$'s are in general position is to ensure that $\mathcal Z$ has complex 
codimension at least two. This is what allows us to remove indeterminacies in the 
smooth category (as opposed to the algebraic or complex analytic categories, where 
complex codimension greater than one follows naturally). Any hypothesis that 
ensures that the indeterminacy locus has complex codimension greater than one would 
suffice.
\end{remark}

\section*{Acknowledgements}

We are extremely grateful to the referee for very detailed and helpful comments and 
particularly for Remark \ref{referee} that considerably simplifies of the 
proof of the main theorem. Both authors acknowledge support of J. C. Bose Fellowship.

\end{document}